\documentclass{article}
\usepackage[english]{babel}
\usepackage[utf8x]{inputenc}
\usepackage{amssymb}
\usepackage{amsmath}
\usepackage{mathrsfs}
\usepackage{bm}
\usepackage{amsthm}
\usepackage{enumitem}
\usepackage{color}
\usepackage{graphicx}
\usepackage{bm}

\newtheorem{prop}{Proposition}[section]
\newtheorem{teorema}[prop]{Theorem}
\newtheorem{lemma}[prop]{Lemma}

\newcommand{\g}{\mathfrak}
\newcommand{\R}{\mathbf{R}}
\newcommand{\C}{\mathbf{C}}
\renewcommand{\H}{\mathbf{H}}
\newcommand{\per}{\cdot}
\newcommand{\gr}[1]{{\bf#1}}
\newcommand{\alfa}{\alpha}
\newcommand{\de}{\partial}
\newcommand{\D}{\Delta}
\newcommand{\Z}{\mathbf{Z}}
\newcommand{\ci}[1]{\mathscr{#1}}
\newcommand{\ecc}{\ldots}

\renewcommand{\div}{\operatorname{div}}
\newcommand{\grad}{\nabla}
\renewcommand{\l}{\lambda}
\renewcommand{\o}{\omega}
\newcommand{\bra}{\left\langle}
\newcommand{\ket}{\right\rangle}
\renewcommand{\phi}{\varphi}
\renewcommand{\d}{\delta}
\newcommand{\e}{\varepsilon}
\newcommand{\N}[1]{\left\lVert#1\right\rVert}
\renewcommand{\O}{\Omega}

\newcommand{\mi}{\mu}
\renewcommand{\ni}{\nu}

\DeclareMathOperator{\supp}{supp}

\title{Singular periodic solutions to a critical equation in the Heisenberg group}
\author{Claudio Afeltra\footnote{Scuola Normale Superiore, Piazza dei Cavalieri 7, 56126 Pisa (Italy) - claudio.afeltra@sns.it}}
\date{}

\begin{document}

\maketitle

\begin{abstract}
 We construct positive solutions to the equation
 $$-\D_{\H^n} u = u^{\frac{Q+2}{Q-2}}$$
 on the Heisenberg group, singular in the origin, similar to the Fowler solutions of the Yamabe equations on $\R^n$. These satisfy the homogeneity property $u\circ\d_T=T^{-\frac{Q-2}{2}}u$ for some $T$
 large enough, where $Q=2n+2$ and $\d_T$ is the natural dilation in $\H^n$.
 We use the Lyapunov-Schmidt method applied to a family of approximate solutions built by periodization from the global regular solution
 classified in \cite{JL}.
\end{abstract}

 

\section{Introduction}
 Let $\H^n$ be the Heisenberg group with its standard pseudohermitian structure. The problem of studying constant \emph{Webster curvature} pseudohermitian
 structures conformal to the standard one, in the spirit of the classical Riemannian case, is equivalent to find the positive solutions of the equation
 \begin{equation}\label{Equazione}
  -\D_{\H^n} u = u^{\frac{Q+2}{Q-2}},
 \end{equation}
 where $\D_{\H^n}$ is the sublaplacian and $Q=2n+2$ is the homogeneous dimension (in Section \ref{notazioni} we will recall the preliminary
 definitions about the Heisenberg group).
 
 The positive solutions of equation \eqref{Equazione} satisfying some integrability hypotheses were classified by Jerison and Lee \cite{JL},
 and they correspond to conformal factors that trasform the standard pseudohermitian structure of $\H^n$ into the push-forward of the
 pseudohermitian structure of the sphere $\gr{S}^{2n+1}\subset\C^{n+1}$ with respect to the Cayley transform, up to translations and dilations.
 This classification plays an important role in the solution
 of the \emph{CR Yamabe problem}, see \cite{JL2}, \cite{GamYac}, \cite{Gam} and \cite{CMY}.
 
 In the Euclidean space the analogous equation,
 \begin{equation}\label{EquazioneEuclidea}
  -\D_{\H^n} u = u^{\frac{n+2}{n-2}},
 \end{equation}
 is well studied, being related to the Yamabe problem, and being analytically interesting due to a lack of compactness.
 
 The Yamabe equation in $\R^n$ also arises when looking for extremals of the critical Sobolev-Gagliardo-Nirenberg inequality.
 These were classified as ``bubble functions'' independently by Aubin \cite{Au} and Talenti \cite{Ta}.
 
 A complete classification for solutions of \eqref{EquazioneEuclidea} (without integrability hypotheses)
 was given by Caffarelli, Gidas and Spruck \cite{CGS}.
 In the this case also the solutions on $\R^n\setminus\{0\}$ were classified. In addition to the regular ones on the whole space,
 there is a singular solution corresponding geometrically to the cylindrical metric, and a family of \emph{singular solutions}, the Fowler solutions,
 which correspond to a family of periodic metrics on the cylinder which are isometric to the Delaunay surfaces
 (see \cite{MP1} and the references cited therein).
 
 This terminology is in analogy with the structure structure for axially symmetric constant mean curvature surfaces: in this case Delaunay surfaces
 bridge the sphere and the cylinder (see \cite{MP2} and the references cited therein).
 Furthermore, the Fowler solutions have been used as building blocks (see for example \cite{MP1}) for the construction of more general solutions
 (as well as for the constant mean curvature Delaunay surfaces).
 
 
 The above classification has been used to study the asymptotic profiles of general singular solution (see \cite{KMPS}), and solutions with singular behavior as the Fowler's ones arise in the
 study of blow-up limits for the prescribed scalar curvature problem (see \cite{L1}, \cite{L2}, \cite{CL}).
  
 The aim of this article is to prove, in analogy with the Euclidean case, the existence of a family of solutions to equation \eqref{Equazione}
 satisfying a periodicity condition with respect to dilations, that is such that $u\circ \delta_T = T^{\alfa}u$ for some $T,\alfa$ (see Section \ref{notazioni}).
 for the notation). A simple computation shows that it must necessarily hold that $\alfa=-\frac{Q-2}{2}$.
 The main result of the paper is the following.
 
 \begin{teorema}\label{Teorema}
  There exists $T_0$ such that for $T\ge T_0$ there exists a positive solution of the equation \eqref{Equazione}
  on $\H^n\setminus\{0\}$ such that
  $u\circ \delta_T = T^{-\frac{Q-2}{2}}u,$
  and $T$ is the smallest period.
 \end{teorema}
 
 On the Euclidean space the proof of the uniqueness of such solutions relies on a result (in \cite{CGS}), proved by the moving planes method,
 stating that the positive solutions of equation \eqref{Equazione} are radially
 symmetric. In this way the construction of solutions and their classification is carried out by a standard ODE analysis. This cannot be done on $\H^n$. 
 We point out that on the Heisenberg group one cannot expect a symmetric solution, because the sublaplacian is not rotationally invariant.
 We also point out the recent results in \cite{GMM}, where solutions with singularities at higher-dimensional sets were constructed with different methods. 

 Theorem \ref{Teorema} is proved by writing equation \eqref{Equazione} as the variational equation of the functional
 $$\ci{J}_T(u)=\int_{\Omega_T}\left(|\grad_{\H^n} u|^2-\frac{1}{2^*}|u|^{2^*}\right)$$
 on a space of functions satisfying $u\circ\d_T=T^{-\frac{Q-2}{2}}u$
 (the integral is with respect to the Haar measure, see Section \ref{notazioni}).
 
 In Section 3 we will find an estimate of the Sobolev constant for periodic solution through a Hardy-Littlewood-Sobolev
 type theorem for Lorentz spaces. This will be used to carry out the estimates in the subsequent sections.
 
 In Section \ref{Costruzione} we will build a family $\ci{Z}_T$ of approximate critical points of $\ci{J}_T$ by gluing
 a sequence of suitable dilations of the global regular solution $\o_{\l}$.
 We will show that these solutions are ``almost critical'' points, in the sense that on $\ci{Z}_T$ the differential of the functional $\ci{J}_T$ is small.
 
 In Section \ref{DifferenzialeSecondo} we will prove that a non degeneracy condition holding for $\o_{\l}$ can be carried on $\Psi_{\l}$.
 
 In the final Section we will prove the existence of the desired solutions through the Lyapunov-Schmidt method, reducing the problem to the orthogonal of the tangent of the curve $\ci{Z}_T$,
 and therein applying the contraction Theorem.
 
 We believe that the construction should give perspectives for the study of more general singular solutions in the Heisenberg group, in the spirit of the cited results on the Euclidean space.

\section{Preliminaries and notation}\label{notazioni}
 In this section we recall some basic definitions and facts on the Heisenberg group, widely present in the literature. See, for example, Chapter 10 of \cite{CS}.
 
 Let us consider the Heisenberg group $\H^n=\C^n\times \R$, with the convention on the product
 $$(z_1,t_1)\per (z_2,t_2)=(z_1+z_2, t_1+t_2+2\,\g{Im}(z_1\per\overline{z_2})).$$
 Let
 $$X_i=T_i=\frac{\de}{\de x_i}+2y_i\frac{\de}{\de t},$$
 $$Y_i=T_{n+i}=\frac{\de}{\de y_i}-2x_i\frac{\de}{\de t},$$
 $$T=T_0=\frac{\de}{\de t}$$
 be the standard basis of left invariant vector fields,
 $$\grad_{\H^n} u = \sum_i X_i(u)X_i+Y_i(u)Y_i$$
 the subriemannian gradient,
 $$\div\left(\sum_i f_iX_i+g_iY_i\right)=\sum_iX_i(f_i)+Y_i(g_i)$$
 the divergence (which coincides with the divergence with respect to a Haar volume form), and
 $$\D_{\H^n}=\div\circ \grad_{\H^n} =\sum_i X_i^2+Y_i^2$$ 
 be the sublaplacian.
 There exists a constant $C=C(n)$ such that
 \begin{equation}\label{SoluzioneFondamentale}
  K(x) = \frac{C}{|x|^{Q-2}}
 \end{equation}
 is a fundamental solution of the sublaplacian.
 Let
 $$S^1(\H^n)= \{ u\in L^2(\H^n) \;|\; X_iu,Y_iu\in L^2(\H^n)\}.$$
 We endow $\H^n$ with the set of dilations
 $$\delta_{\l}(z,t)=(\l z, \l^2t)$$
 and with the homogeneous norm
 $$|(z,t)|=\left(|z|^4+t^2\right)^{1/4}.$$
 
 The Lebesgue measure $dx$ is a biinvariant Haar measure on $\H^n$ satisfying
 $$(\d_{\l})_{\#}dx=\l^Q dx;$$
 this is essentially the reason why $Q$ takes the place of the topological dimension $n$ in many analytic questions.
 
 
 Let us set $B_r=\{ |x|<r\}$ and $\Omega_T=\overline{B}_R\setminus B_1$.
 We define the Hilbert space
 $$X_T=\{u\in S^1_{\rm{loc}}(\H^n) \;|\; u\circ \delta_T = T^{-\frac{Q-2}{2}}u\}$$
 with the product
 $$\bra u, v\ket = \int_{\Omega_T}\grad_{\H^n} u\per\grad_{\H^n} v.$$
 Let $\widetilde{X}_T$ the closed subspace of $X_T$ of the functions of the form $u(|z|,t)$.
 
 It is known that the positive solutions of the equation \eqref{Equazione} are
 $$\o_{\l}=\l^{(2-Q)/2}\o\circ\delta_{\l^{-1}}$$
 and the translates thereof, where
 $$\omega(z,t)=c_0\frac{1}{\left( t^2 + (1+|z|^2)^2\right)^{(Q-2)/4}}.$$
 
 The problem is variational: the solutions in $S^1(\H^n)$ of the equation are the critical points of the functional
 $$\ci{J}(u)=\int_{\H^n}\left(|\grad_{\H^n} u|^2-\frac{1}{2^*}|u|^{2^*}\right).$$
 Analogously the solutions of the equation of the equation in $X_T$ are the critical points of the functional
 $$\ci{J}_T(u)=\int_{\Omega_T}\left(|\grad_{\H^n} u|^2-\frac{1}{2^*}|u|^{2^*}\right).$$
 It holds that
 $$d\ci{J}_T(u)[\phi]=\int_{\Omega_T}\grad_{\H^n} u\per\grad_{\H^n}\phi - u|u|^{2^*-2}\phi$$
 and that
 $$d^2\ci{J}_T(u)[\phi,\psi]= \int_{\Omega_T}\grad_{\H^n}\phi\per\grad_{\H^n}\psi -(2^*-1)|u|^{2^*-2}\phi\psi.$$
 We call $\ci{J}''_T$ the operator associated with this bilinear form in the natural way:
 $$\bra \ci{J}''_T(u)[\phi],\psi \ket = d^2\ci{J}_T(u)[\phi,\psi].$$
 
 Let us notice that, if $u\in X_T$ and $E\subseteq \gr{H}^n$ then
 \begin{equation}\label{DilatazioneLq}
  \int_{\d_r(E)}|u|^{\frac{2Q}{Q-2}} = \int_E |r^{\frac{Q-2}{2}}u\circ\d_r|^{\frac{2Q}{Q-2}}
 \end{equation}
 and
 \begin{equation}\label{DilatazioneGradiente}
  \int_{\d_r(E)}|\grad_{\H^n} u|^2 = \int_E |r^{\frac{Q-2}{2}}\grad_{\H^n}(u\circ\d_r)|^2.
 \end{equation}
 In particular, if $1\le r\le T$ then
 \begin{equation}\label{CambioDominioLq}
  \begin{aligned}
   \int_{\d_r\O_T} |u|^{\frac{2Q}{Q-2}} = \int_{\O_T\setminus\O_r}|u|^{\frac{2Q}{Q-2}}+\int_{\O_{rT}\setminus\O_T}|u|^{\frac{2Q}{Q-2}} =\\
   = \int_{\O_T\setminus\O_r}|u|^{\frac{2Q}{Q-2}}+\int_{\O_{r}}|T^{\frac{Q-2}{2}}u\circ\d_T|^{\frac{2Q}{Q-2}} = \int_{\O_T}|u|^{\frac{2Q}{Q-2}},
  \end{aligned}
 \end{equation}
 and by induction and inversion one can extend this formula to every value of $r$. Analogously
 \begin{equation}\label{CambioDominioGradiente}
  \int_{\d_r\O_T}|\grad_{\H^n} u|^2 = \int_{\O_T}|\grad_{\H^n} u|^2,
 \end{equation}
 and by polarization
 \begin{equation}\label{CambioDominioProdottoScalare}
  \int_{\d_r\O_T}\grad_{\H^n} u\per \grad_{\H^n} v = \int_{\O_T}\grad_{\H^n} u\per \grad_{\H^n} v.
 \end{equation}
 
 The following lemma shows that in integration by parts in $X_T$ boundary terms are null.
 
 \begin{lemma} 
  If $u,v\in X_T$ then
  $$\int_{\O_T}\grad_{\H^n} u\per\grad_{\H^n} v=-\int_{\O_T}\D_{\H^n} u\per v.$$
 \end{lemma}

 \begin{proof}
  Let us write $v=v_1+v_2$ with $u_1,u_2\in X_T$, $\supp u_1\cap \O_T \subset B_{7(T+1)/8}\setminus B_{(T+1)/8}$ and
  $\supp u_2\cap\d_{(T+1)/2}\O_T \subset B_{(T+1)^2/8}\setminus B_{3(T+1)/4}$ (this can be carried out through a partition of unity).
  
  Then, using formula \eqref{CambioDominioProdottoScalare},
  $$\int_{\O_T}\grad_{\H^n} u\per\grad_{\H^n} v= \int_{\O_T}\grad_{\H^n} u\per\grad_{\H^n} v_1 + \int_{\O_T}\grad_{\H^n} u\per\grad_{\H^n} v_2 =$$
  $$= -\int_{\O_T}\D_{\H^n} u\per v_1 + \int_{\d_{(T+2)/2}\O_T}\grad_{\H^n} u\per\grad_{\H^n} v_2 =$$
  $$= -\int_{\O_T}\D_{\H^n} u\per v_1 - \int_{\d_{(T+2)/2}\O_T}\D_{\H^n} u\per v_2=$$
  $$ = -\int_{\O_T}\D_{\H^n} u\per v_1 - \int_{\O_T}\D_{\H^n} u\per v_2 = -\int_{\O_T}\D_{\H^n} u\per v.$$
 \end{proof}

 We will need to restrict ourselves to solutions in $\widetilde{X}_T$. In order to do this, we observe that, under the identification
 $\H^n=\R^{2n}\times\R$, the functional $\ci{J}_T$ is invariant by the group of transformations of the form $(z,t)\mapsto(Az,t)$
 with $A\in O(\R^{2n})\cap Sp(\R^{2n})$. In fact it is known that if $A=(a_{ij})\in Sp(\R^{2n})$ then this transformation is a group automorfism of
 $\H^n$ (see \cite{Folland}, Chapter 1, Section 2), and so it maps the fields $T_i$ into the fields $\sum_j a_{ij}T_j$.
 So, using the fact that $A\in O(\R^{2n})$, it is easy to verify that $\ci{J}_T$ is invariant by this group.
 
 Furthermore, under the canonical identification of $\R^{2n}$ with $\C^n$, $O(\R^{2n})\cap Sp(\R^{2n})=U(\C^n)$ (\cite{Folland}, Proposition 4.6).
 
 Since $U(n)$ acts transitively on the unit sphere of $\C^n$, $\widetilde{X}_T$ is the set of the functions in $X_T$ invariant under the 
 transformations of this form, and so, by
 Palais' criticality principle \cite{Palais}, the critical points of the restriction of $\ci{J}_T$ to it are critical points in all of $X_T$.
 
 In the sequel we will need also a particular vector field that plays an important role in $\H^n$ (and more in general in homogeneous groups), the generator of the dilations.
 It is characterized by the equation
 $$ \left.\frac{d}{d\l}\right|_{\l=1}(u\circ\d_{\l})=Zu$$
 for every $u\in\ci{C}^1(\H^n)$. An explicit expression for it is
 $$Z= \sum_{i=1}^n x_i\frac{\de}{\de x_i} + y_i\frac{\de}{\de y_i} + 2t\frac{\de}{\de t}.$$
 It is easy to verify that
 \begin{equation}\label{GeneratoreDilatazioni}
   \l\frac{d}{d\l}(u\circ\d_{\l}) = Z(u\circ\d_{\l}) = (Zu)\circ\d_{\l}.
 \end{equation}
 Using this formula, it is easy to prove that a function $u$ is homogeneous of degree $\alfa$ if and only if $Zu=\alfa u$
 (an extension to $\H^n$ of Euler's theorem).
 
 Furthermore it holds that $[X_i,Z]=X_i$ and $[Y_i,Z]=Y_i$, and so $[\grad_{\H^n}, Z]=\grad_{\H^n} $.
 
 \subsection{Lorentz spaces}
 In Section 3, to overcome the non integrability of the functions in $X_T$ in the whole space, will need to use the Lorentz spaces, which we recall briefly.

 Given a $\sigma$-finite measure space $(X,\mi)$ and $1\le p<\infty$, $1\le q\le\infty$, the Lorentz quasinorm is defined as
 $$\N{u}_{L^{p,q}(X)} = p^{1/q} \N{ \l \mi\{|u|>\l\}^{1/p} }_{L^q(dt/t)}.$$
 Furthermore one defines $\N{u}_{L^{\infty,\infty}(X)}=\N{u}_{L^{\infty}(X)}$.
 The Lorentz space $L^{p,q}(X)$ is the set of functions such that this quantity is finite. When $p=q$, $\N{u}_{L^{p,p}}=\N{u}_{L^p}$, while when $q=\infty$, $L^{p,\infty}$ coincides with weak $L^p$.

 We will need the following generalization of the Young inequality, which sometimes is referred in the literature as Young-O’Neil inequality.
 It can be deduced applying Theorem 2.6 in \cite{ON} (with the corrections in \cite{Yap}) and Theorem 1.2.12, Remark 1.2.11 in \cite{Gr}.

 \begin{teorema}\label{YoungONeil}
  If $1<p,p_1,p_2<\infty$, $1\le q,q_1,q_2\le\infty$ are such that
  $$ \frac{1}{p_1} + \frac{1}{p_2} = 1 + \frac{1}{p} \;\;\; \text{and} \;\;\; \frac{1}{q_1} + \frac{1}{q_2} = \frac{1}{q}$$
  then there exists $C$ such that for every $f\in L^{p_1,q_1}(\H^n)$, $g\in L^{p_2,q_2}(\H^n)$ it holds
  $$ \N{fg}_{L^{p,q}(\H^n)} \le C \N{f}_{L^{p_1,q_1}(\H^n)} \N{g}_{L^{p_2,q_2}(\H^n)}.$$
 \end{teorema}
 
 \subsection{Basic definitions on CR geometry}
 For convenience of the reader, we recall the basic definitions about CR manifolds, also if we will not use them.
 The reader can find more on the topic in \cite{Bog}, \cite{DTom}.
 
 A CR manifold is a real smooth manifold $M$ endowed with a subbundle $\ci{H}$ of the complexified tangent bundle of $M$, $T^{\C}M$, such that $\ci{H}\cap\overline{\ci{H}} = \{0\}$
 and $[\ci{H},\ci{H}]\subseteq\ci{H}$. We will assume $M$ to be of hypersurface type, that is that $\dim M=2n+1$ and that $\dim\ci{H}=n$.
 There exists a non-zero real differential form $\theta$ that is zero on $\g{Re}(\ci{H}\oplus\overline{\ci{H}})$;
 it is unique up to scalar multiple by a function.
 Such a form is called pseudohermitian structure. On a pseudohermitian manifold, the Levi form is defined as the 2-form on $\ci{H}$
 $L_{\theta}(V,W)=-id\theta(V,\overline{W})=id\theta([V,\overline{W}])$. A CR manifold is said to be pseudoconvex if it admits a positive definite
 Levi form (this implies every Levi form to be definite).
 
 The Heisenberg group is the simplest pseudoconvex CR manifold, if endowed with the bundle
 $\ci{H}=\operatorname{span}(Z_1,\ecc,Z_n)$ with $Z_j=\frac{1}{2}(X_j-iY_j)$.
 
 On a nondegenerate pseudohermitian manifold one can define a connection, the Tanaka-Webster connection. This allows to define
 curvature operators in an analogous manner as in Riemannian geometry: the pseudohermitian curvature tensor is the curvature of the
 Tanaka-Webster connection, the Ricci tensor is
 $$\operatorname{Ric}(X,Y)=\operatorname{trace}(Z\mapsto R(Z,X)Y),$$
 and the Webster scalar curvature is the trace of the Ricci tensor with respect to the Levi form.
 
 Being a pseudohermitian structure defined only up to a conformal factor on a CR manifold, in CR geometry the Yamabe problem is even more more
 natural than in Riemannian geometry. If $\widetilde{\theta}=u^{2/n}\theta$, the transformation law of the Webster curvature is
 $$\widetilde{W}=u^{-1-2/n}\left(\frac{2n+2}{n}\D_bu + Wu\right),$$
 where $\D_b$ is the sublaplacian, which can be defined in a similar way as the Heisenberg group.
 So the Yamabe problem takes to the equation
 $$\frac{2n+2}{n}\D_bu + Wu=\l u^{1+2/n}.$$
 Since the Heisenberg group has zero Webster curvature, and since the pseudohermitian sublaplacian coincides with the sublaplacian defined
 formerly, the Yamabe problem, up to an inessential constant, is equivalent to find positive solution to equation \ref{Equazione}.
 
 The solution of this case plays in the solution in the general case the same role that the solution on $\R^n$ plays in the solution
 of the general Riemannian case.
 
 \section{Estimate of the Sobolev constant on $X_T$}
 
 In order to carry out the estimates in the next Sections, we will need an explicit bound on the  Sobolev constant on $X_T$. We will achieve this relating the $L^p$ norm on $\O_T$ and the $L^{p,\infty}$ norm on the whole space.

 \begin{prop}\label{NormeEquivalenti}
  If $f$ is an $L^p_{loc}$ function on $\H^n\setminus\{0\}$ such that $f\circ\d_T=T^{-\alfa}f$ and $\alfa p=Q$ then
  $$ \left(\frac{T^Q-1}{T^Q}\right)^{1/p} \N{u}_{L^{p,\infty}(\H^n)} \le C_2 \N{u}_{L^p(\O_T)} \le Q^{1/p}(\log T)^{1/p} \N{u}_{L^{p,\infty}(\H^n)}.$$
 \end{prop}

  \begin{proof}
  Let us call $f(\l)=\mi\{x\in \O_T \;|\; u(x)>\l\}$ and $g(t)=\mi\{x\in \H^n \;|\; u(x)>\l\}$. Then it holds that
  $$g(\l) = \sum_{k\in\Z} T^{Qk} f(\l T^{\alfa}).$$
  Therefore for every $\l>0$
  $$\N{u}_{L^p(\O_T)}^p = p\int_0^{\infty}\xi^{p-1}f(\xi)d\xi =  p\sum_{k\in \Z} \int_{\l T^{\alfa (k-1)}}^{\l T^{\alfa k}} \xi^{p-1} f(\xi) d\xi \ge$$
  $$\ge p\sum_{k\in \Z} f(\l T^{\alfa k}) \int_{\l T^{\alfa (k-1)}}^{\l T^{\alfa k}} \xi^{p-1} d\xi = \frac{T^Q-1}{T^Q} \l^p \sum_{k\in \Z} T^{Qk} f(\l T^{\alfa k}) =$$
  $$= \frac{T^Q-1}{T^Q} \l^p g(\l).$$
  Taking the supremum with respect to $\l$ we get the first inequality.

  For the other one, let us pick an integer $N>0$ and write
  $$\N{u}_{L^p(\O_T)}^p = p\int_0^{\infty}\xi^{p-1}f(\xi)d\xi =  p\sum_{k\in \Z} \int_{ T^{\alfa k/N}}^{ T^{\alfa (k+1)/N}} \xi^{p-1} f(\xi) d\xi \le$$
  $$ \le p\sum_{k\in \Z} f(T^{\alfa k/N}) \int_{ T^{\alfa k/N}}^{ T^{\alfa (k+1)/N}} \xi^{p-1} d\xi =$$
  $$= \sum_{m=1}^N \sum_{j\in\Z} (T^{\alfa p/N}-1) T^{\alfa pj} T^{\alfa pm/N}f(T^{\alfa j}T^{\alfa m/N}) = $$
  $$ = (T^{Q/N}-1) \sum_{m=1}^N T^{Qm/N} \sum_{j\in\Z} T^{Qj} f(T^{\alfa j}T^{\alfa m/N}) = $$
  $$ =  (T^{Q/N}-1) \sum_{m=1}^N T^{Qm/N} g(T^{\alfa m/N}) \le N(T^{Q/N}-1) \N{u}_{L^{p,\infty}}^p .$$
  Taking the limit for $N\to\infty$ we get the second inequality.
 \end{proof}

  Using the Theorem \ref{YoungONeil} we can prove a Sobolev type inequality for weak $L^p$ spaces.

 \begin{prop}\label{SobolevLpDebole}
  There exists a constant $C$ such that for every function $u\in L^{2,\infty}(\H^n)$ such that $\grad u\in L^{2,\infty}(\H^n)$ verifies
  $$ \N{u}_{L^{\frac{2Q}{Q-2},\infty}} \le C \N{\grad u}_{L^{2,\infty}}.$$
 \end{prop}

 \begin{proof}
  Let $E= { u> 1}$, $E^c = \H^n\setminus E$, $u_1 = u\chi_{E^c}+\chi_E$ and $u_2=(u-1)\chi_E$, so that $u=u_1+u_2$.
  It is standard to prove that $u_1$ and $u_2$ have weak subriemannian gradient and that $\grad_{\H} u_1 = (\grad_{\H} u)\chi_{E^c}$,
  $\grad_{\H} u_2 = (\grad_{\H} u)\chi_{E}$ (the proof is the same as on $\R^n$).
  It is easy to prove that $u_1 \in S^p(\H^n)$ for $p>2$ and that $u_2\in S^q(\H^n)$ for $q<2$.
  If $\phi\in\ci{C}^{\infty}_c(\H^n)$ it holds that
  $$\phi(x)= (\phi*\delta)(x)= (\phi*(-\D_{\H^n}K))(x) =$$
  \begin{equation}\label{Nucleo}
  = \int_{\H^n}(\grad_{\H^n}\phi)(xy^{-1}) * (\grad_{\H^n}K)(y)dy := (\grad_{\H^n}\phi *\grad_{\H^n}K) (x)
  \end{equation}
  Formula \eqref{SoluzioneFondamentale} implies that $\grad_{\H^n}K\in L^{\frac{Q}{Q-1},\infty}$, and so,
  by Theorem \ref{YoungONeil}, the operator $f\mapsto f*\grad_{\H^n}K$ is bounded from $L^p$ and $L^q$ to some other Lebesgue spaces.
  Therefore, using the density of $\ci{C}^{\infty}_c$ in $S^p(\H^n)$ for $1\le p<\infty$, formula \eqref{Nucleo} holds almost everywhere
  for functions in these spaces, and so it holds for $u_1$ and $u_2$. By summing one obtains that
  $$ u = \grad_{\H^n}u * \grad_{\H^n}K.$$
  The thesis follows applying Theorem \ref{YoungONeil} once more.
 \end{proof}

 We point out that in the proof of the last Proposition the splitting of $u$ in two pieces belonging to some $L^p$ space was necessary
 because $\ci{C}^{\infty}_c$ functions are not dense in the weak $L^p$ spaces.
 
 Combining Propositions \ref{NormeEquivalenti} and \ref{SobolevLpDebole} we get the following Sobolev theorem for $X_T$ spaces
 with an explicit constant.
 
 \begin{prop}\label{lemmaSobolev}
  There exist a constant $C$ independent by $T$ such that for every $u\in X_T$
  $$ \N{u}_{L^{\frac{2Q}{Q-2}}(\O_T)} \le C (\log T)^{\frac{Q-2}{2Q}} \left(\frac{T^Q}{T^Q-1}\right)^{1/2} \N{u}_{X_T}.$$
 \end{prop}
 
\section{Construction of a family of approximate solutions}\label{Costruzione}
 In order to apply a perturbative method, we find a family of approximate stationary points of $\ci{J}_T$ for $T$ big enough.
 
 The family is the following:
 
 $$\Psi_{\l,T} = \sum_{k\in\gr{Z}} \o_{\l/T^k} = \sum_{k\in\gr{Z}}T^{\frac{Q-2}{2}k}\o_{\l}\circ\d_{T^k}$$
 (we will hide the dependence by $T$ whether not necessary).
 The series converges uniformly on compact sets, because, if $x\in K$,
 $$\Psi_{\l}(x) = \sum_{k\in\gr{Z}}T^{\frac{Q-2}{2}k}\o_{\l}\circ\d_{T^k} \le$$
 $$\le C_{\l,K}\sum_{k\ge 0}T^{\frac{Q-2}{2}k} \frac{1}{T^{k(Q-2)}} + C_{\l,K}\sum_{k<0}T^{\frac{Q-2}{2}k} \le C_{\l,K}.$$
 The subriemannian gradient satisfies
 $$|\grad_{\H^n}\Psi_{\l}(x)| \le \sum_{k\in\gr{Z}}T^{\frac{Q-2}{2}k}T^k|\grad_{\H^n}\o_{\l}|\circ\d_{T^k} \le$$
 $$\le C_{\l,K}\sum_{k\ge 0}T^{\frac{Q}{2}k} \frac{1}{T^{k(Q-1)}} + C_{\l,K}\sum_{k<0}T^{\frac{Q}{2}k} \le C_{\l,K}$$
 and so it converges uniformly on compact sets. The same holds for higher order subriemannian derivatives.
 $\Psi_{\l}\in X_T$ because
 $$\Psi_{\l}\circ\d_T = \sum_{k\in\gr{Z}}T^{\frac{Q-2}{2}k}\o_{\l}\circ\d_{T^k}\circ\d_T = T^{-\frac{Q-2}{2}}\sum_{k\in\gr{Z}}T^{\frac{Q-2}{2}k}\o_{\l}\circ\d_{T^k}=T^{-\frac{Q-2}{2}}\Psi_{\l}.$$
 It holds that
 $$\Psi_{T\l} = \sum_{k\in\gr{Z}}T^{\frac{Q-2}{2}k}\o_{T\l}\circ\d_{T^k} = \sum_{k\in\gr{Z}} T^{\frac{Q-2}{2}k} \frac{1}{(T\l)^{\frac{Q-2}{2}}}\o\circ\d_{1/T\l}\circ\d_{T^k}=$$
 $$= \sum_{k\in\gr{Z}} T^{\frac{Q-2}{2}(k-1)}\frac{1}{\l^{\frac{Q-2}{2}}}\o\circ\d_{1/\l}\circ\d_{T^{k-1}}= \sum_{k\in\gr{Z}} T^{\frac{Q-2}{2}(k-1)}\o_{\l}\circ\d_{T^{k-1}} = \Psi_{\l}.$$
 Therefore the set $\ci{Z}_T=\{ \Psi_{\l} \;|\; \l\in (0,\infty)\}$ is a closed curve in $X_T$.
 
 Moreover, using formula \eqref{GeneratoreDilatazioni}, it can be computed that
 
 $$\frac{\de\Psi_{\l}}{\de\l} = \frac{\de}{\de\l}\sum_{k\in\gr{Z}} \o_{\l/T^k} = \sum_{k\in\gr{Z}}\frac{\de}{\de\l}\left(\l^{-\frac{Q-2}{2}}\o_{1/T^k}\circ\d_{\l^-1}\right)=$$
 $$= \sum_{k\in\gr{Z}}\left( -\frac{Q-2}{2}\frac{1}{\l}\o_{\l/T^k} - \l^{-\frac{Q-2}{2}}\frac{1}{\l^2}\l Z(\o_{1/T^k}\circ\d_{\l^-1})\right)=$$
 $$= \sum_{k\in\gr{Z}}\left( -\frac{Q-2}{2}\frac{1}{\l}\o_{\l/T^k} - \frac{1}{\l}Z(\o_{\l/T^k})\right)=$$
 \begin{equation}\label{DerivataLambda}
  = -\frac{Q-2}{2}\frac{1}{\l}\Psi_{\l} - \frac{1}{\l}Z(\Psi_{\l}).
 \end{equation}

 This implies that the curve $\ci{Z}_T$ is immersed for $T$ big enough, because if $\frac{\de\Psi_{\l}}{\de\l}$ was zero then
 $Z(\Psi_{\l})=-\frac{Q-2}{2}\Psi_{\l}$ would be zero, and by the aforementioned Euler's theorem $\Psi_{\l}$ would be homogeneous of
 degree $-\frac{Q-2}{2}$; but it is clearly not by construction if $T$ is big enough.
 
 We want to prove the following proposition.
 
 \begin{prop}\label{StimaGradiente}
  For every $\e$ there exists $T_0$, depending only on by $n$, such that if $T\ge T_0$ then $\N{\grad_{\H^n} \ci{J}_T} <\e$ on $\ci{Z}_T$.
 \end{prop}
 
 We divide the proof in several lemmas.

 First we compute the differential of $\ci{J}_T$ in $\Psi_{\l}$:
 $$d\ci{J}_T(\Psi_{\l})[u]= \int_{\O_T}\grad_{\H^n}\Psi_{\l}\per\grad_{\H^n} u - \Psi_{\l}^{2^*-1}u =$$
 $$= \int_{\O_T}\sum_{k\in\gr{Z}}\grad_{\H^n}\o_{\l/T^k}\per\grad_{\H^n} u- \left(\sum_{k\in\gr{Z}} \o_{\l/T^k}\right)^{2^*-1}u =$$
 $$= \sum_{k\in\gr{Z}}\left( \int_{\O_T}\grad_{\H^n}\o_{\l/T^k}\per\grad_{\H^n} u-\o_{\l/T^k}^{\frac{Q+2}{Q-2}}u\right) + $$
 $$- \int_{\O_T}\left[\left(\sum_{k\in\gr{Z}} \o_{\l/T^k}\right)^{\frac{Q+2}{Q-2}}-\sum_{k\in\gr{Z}}\o_{\l/T^k}^{\frac{Q+2}{Q-2}}\right]u=$$
 \begin{equation}\label{ApiuB}
  :=A+B.
 \end{equation}
 
 \begin{lemma}
  In the above notation, $A=0$.
 \end{lemma}

 \begin{proof} We have
 $$A = \sum_{k\in\gr{Z}}\int_{\O_T}T^{\frac{Q-2}{2}k}(\grad_{\H^n}\o_{\l})\circ\d_{T^k}\per\grad_{\H^n} u-(T^{\frac{Q-2}{2}k})^{\frac{Q+2}{Q-2}}(\o_{\l}\circ\d_{T^k})^{\frac{Q+2}{Q-2}}u=$$
 $$= \sum_{k\in\gr{Z}}\int_{\O_T}T^{\frac{Q}{2}k}(\grad_{\H^n}\o_{\l})\circ\d_{T^k}\per\grad_{\H^n} u- T^{\frac{Q+2}{2}k}(\o_{\l}\circ\d_{T^k})^{\frac{Q+2}{Q-2}}u=$$
 $$= \sum_{k\in\gr{Z}}\int_{\d_{T^k}(\O_T)}T^{-kQ}\left[ T^{\frac{Q}{2}k}\grad_{\H^n}\o_{\l}\per(\grad_{\H^n} u)\circ\d_{T^{-k}} - T^{\frac{Q+2}{2}k}\o_{\l}^{\frac{Q+2}{Q-2}}u\circ\d_{T^{-k}}\right]=$$
 $$= \sum_{k\in\gr{Z}}\int_{\d_{T^k}(\O_T)}T^{-\frac{Q}{2}k}T^k\grad_{\H^n}\o_{\l}\per\grad_{\H^n} (u\circ\d_{T^{-k}}) - \o_{\l}^{\frac{Q+2}{Q-2}}u=$$
 $$= \sum_{k\in\gr{Z}}\int_{\d_{T^k}(\O_T)} \grad_{\H^n}\o_{\l}\per\grad_{\H^n} u- \o_{\l}^{\frac{Q+2}{Q-2}}u= \int_{\H^n} \grad_{\H^n}\o_{\l}\per\grad_{\H^n} u- \o_{\l}^{\frac{Q+2}{Q-2}}u$$
%
  Let us pick a family of smooth functions $\phi_{\e,R}$ such that $\phi_{\e,R}\equiv 1$ on $B_R\setminus B_{2\e}$, $\phi_{\e,R}\equiv 0$ on $B_{\e}$ and $\H^n\setminus B_{R+1}$,
  $|\grad_{\H^n} \phi_{\e,R}|\le \frac{C}{\e}$ on $B_{2\e}\setminus B_{\e}$ and $|\grad_{\H^n} \phi_{\e,R}|\le C$ on $B_{R+1}\setminus B_R$.
  Then
  $$A =\lim_{\substack{\e\to 0\\R\to\infty}} \int_{\H^n} \left(\grad_{\H^n}\o_{\l}\per\grad_{\H^n} u- \o_{\l}^{\frac{Q+2}{Q-2}}u\right)\phi_{\e,R} =$$
  $$= \lim_{\substack{\e\to 0\\R\to\infty}} \int_{\H^n} -(\D_{\H^n}\o_{\l}+\o_{\l}^{\frac{Q+2}{Q-2}})u\phi_{\e,R} - u\grad_{\H^n}\o_{\l}\per\grad_{\H^n}\phi_{\e,R} = $$
  $$ =  -\lim_{R\to\infty} \int_{B_{R+1}\setminus B_R} u\grad_{\H^n}\o_{\l}\per\grad_{\H^n}\phi_{\e,R}- \lim_{\e\to 0} \int_{B_{2\e}\setminus B_{\e}} u\grad_{\H^n}\o_{\l}\per\grad_{\H^n}\phi_{\e,R}.$$
  If $x\to\infty$ then $\grad_{\H^n}\o_{\l}\lesssim\frac{1}{|x|^{Q-1}}$ and $u\lesssim|x|^{-\frac{Q-2}{2}}$, and so the first limit is zero.
  If $x\to 0$ then $\grad_{\H^n}\o_{\l}\lesssim 1$ and $u\lesssim|x|^{-\frac{Q-2}{2}}$, and so also the second limit is zero. Therefore $A=0$.
 \end{proof}
 
 Now we have to estimate the term $B$ from formula \eqref{ApiuB}.
 

 

 \begin{lemma}\label{TermineB}
  In the above notation
  $$|B| \le C(T)\N{u}_{X_T},$$
  where $C(T)$ tends to zero uniformly in $\l$ as $T$ tends to infinity.
 \end{lemma}
 
 \begin{proof}
 $$|B| \le \int_{\O_T}\left[\left(\sum_{k\in\gr{Z}} \o_{\l/T^k}\right)^{\frac{Q+2}{Q-2}}-\sum_{k\in\gr{Z}}\o_{\l/T^k}^{\frac{Q+2}{Q-2}}\right]|u|\le$$
 $$\le \int_{\O_T}\left[\left(\sum_{k\in\gr{Z}} \o_{\l/T^k}\right)^{\frac{Q+2}{Q-2}}-\o_{\l}^{\frac{Q+2}{Q-2}}\right]|u| \le$$
 $$\le \left\{\int_{\O_T}\left[\left(\sum_{k\in\gr{Z}} \o_{\l/T^k}\right)^{\frac{Q+2}{Q-2}}-\o_{\l}^{\frac{Q+2}{Q-2}}\right]^{\frac{2Q}{Q+2}}\right\}^{\frac{Q+2}{2Q}}\N{u}_{L^{\frac{2Q}{Q-2}}(\O_T)} \le$$
 $$\le C(\log T)^{\frac{Q-2}{2Q}}\N{u}_{X_T}\left\{\int_{\O_T}\left[\left(\sum_{k\in\gr{Z}} \o_{\l/T^k}\right)^{\frac{Q+2}{Q-2}}-\o_{\l}^{\frac{Q+2}{Q-2}}\right]^{\frac{2Q}{Q+2}}\right\}^{\frac{Q+2}{2Q}} =$$
 $$= C (\log T)^{\frac{Q-2}{2Q}}\N{u}_{X_T}\per$$
 $$\per\left\{\int_{\O_T}\left[\left(\sum_{k\in\gr{Z}} |x|^{\frac{Q-2}{2}}\o_{\l/T^k}\right)^{\frac{Q+2}{Q-2}}-
 (|x|^{\frac{Q-2}{2}}\o_{\l})^{\frac{Q+2}{Q-2}}\right]^{\frac{2Q}{Q+2}}\frac{dx}{|x|^Q}\right\}^{\frac{Q+2}{2Q}}$$
 by Proposition \ref{lemmaSobolev} (taking $T\ge T_0>1$, since we are going to make a limit for $T\to\infty$).
 Let us define $\eta_{\l}=|x|^{\frac{Q-2}{2}}\o_{\l}$. Then
 $$|B| \le C(\log T)^{\frac{Q-2}{2Q}}\N{u}_{X_T}
 \left\{\int_{\O_T}\left[\left(\sum_{k\in\gr{Z}} \eta_{\l/T^k}\right)^{\frac{Q+2}{Q-2}}-\eta_{\l}^{\frac{Q+2}{Q-2}}\right]^{\frac{2Q}{Q+2}}\frac{dx}{|x|^Q}\right\}^{\frac{Q+2}{2Q}}.$$
 By periodicity we can suppose that $\frac{|x|}{\l}\in\left[\frac{1}{\sqrt{T}},\sqrt{T}\right]$, with $\l=\l(x)$.
 The function $\eta_{\l}$ is bounded and tends to zero for $|x|\to 0,\infty$.
 If $k\ge 0$ and $T$ is large enough then $\eta_{\l/T^k}$ satisfies estimates
 $$|\eta_{\l/T^k}(x)|\lesssim \left( \frac{\left(\frac{T^k}{\l}\right)|x|}{1+\left(\frac{T^k}{\l}\right)^2|x|^2} \right)^{\frac{Q-2}{2}}\lesssim \left( T^k\frac{|x|}{\l}\right)^{-\frac{Q-2}{2}} \le
 \left( \frac{1}{T}\right)^{\left(k-\frac{1}{2}\right)\frac{Q-2}{2}}$$
 and
 $$|\eta_{\l/T^{-k}}(x)|\lesssim \left( \frac{\left(\frac{T^{-k}}{\l}\right)|x|}{1+\left(\frac{T^{-k}}{\l}\right)^2|x|^2} \right)^{\frac{Q-2}{2}}\lesssim
 \left(\frac{1}{T^k}\frac{|x|}{\l}\right)^{\frac{Q-2}{2}} \le \left( \frac{1}{T}\right)^{\left(k-\frac{1}{2}\right)\frac{Q-2}{2}}$$
 uniformly in $\l$.
 It is easy to verify that, for $\alfa,\beta\ge 1$ the function
 $$\frac{\left[(x+y)^{\alfa}-x^{\alfa}\right]^{\beta}}{x^{(\alfa-1)\beta}y^{\beta}+y^{\alfa\beta}}$$
 is bounded on $(0,\infty)^2$, and so there exist $C$ such that
 $$\left[(x+y)^{\alfa}-x^{\alfa}\right]^{\beta} \le C(x^{(\alfa-1)\beta}y^{\beta}+y^{\alfa\beta})$$
 for $x,y\ge 0$. Taking
 $$x=\eta_{\l}, \;\;\;\; y= \sum_{k\in\gr{Z}\setminus\{0\}}\eta_{\l/T^k}, \;\;\;\; \alfa=\frac{Q+2}{Q-2} \;\;\;\;\text{and}\;\;\;\; \beta=\frac{2Q}{Q+2}$$
 one gets that
 $$|B| \le C(\log T)^{\frac{Q-2}{2Q}}\N{u}_{X_T}
 \left\{\int_{\O_T}\left[\eta_{\l}^{\frac{8Q}{(Q+2)(Q-2)}}\left(\sum_{k\in\gr{Z}\setminus\{0\}}\eta_{\l/T^k}\right)^{\frac{2Q}{Q+2}}+\right.\right.$$
 $$\left.\left.+\left(\sum_{k\in\gr{Z}\setminus\{0\}}\eta_{\l/T^k}\right)^{\frac{2Q}{Q-2}} \right]\frac{dx}{|x|^Q}\right\}^{\frac{Q+2}{2Q}}.$$
 Let
 $$\O_T^1= \{x\in\O_T\;|\; \l(x) < 1\}$$
 and
 $$\O_T^2= \{x\in\O_T\;|\; \l(x) \ge 1\}.$$
 Then
 $$|B| \le  C(\log T)^{\frac{Q-2}{2Q}}\N{u}_{X_T} \left\{\left(\int_{\O_T^1}+\int_{\O_T^2}\right)\left[ \eta_{\l}^{\frac{8Q}{(Q+2)(Q-2)}}\left(\sum_{k\in\gr{Z}\setminus\{0\}}\eta_{\l/T^k}\right)^{\frac{2Q}{Q+2}}+\right.\right.$$
 $$\left.\left.+\left(\sum_{k\in\gr{Z}\setminus\{0\}}\eta_{\l/T^k}\right)^{\frac{2Q}{Q-2}} \right]\frac{dx}{|x|^Q}\right\}^{\frac{Q+2}{2Q}}\lesssim$$
 $$\lesssim  C(\log T)^{\frac{Q-2}{2Q}}\N{u}_{X_T}\left\{\int_{\O_T}\left[\left(\frac{1}{T}\right)^{\frac{Q-2}{4}\per \frac{2Q}{Q+2}} +
 \left(\frac{1}{T}\right)^{\frac{Q-2}{4}\per \frac{2Q}{Q-2}}\right]\frac{dx}{|x|^Q}\right\}^{\frac{Q+2}{2Q}}\lesssim$$
 $$\lesssim  C(\log T)^{\frac{Q-2}{2Q}}\N{u}_{X_T}\left\{\left(\frac{1}{T}\right)^{\frac{Q(Q-2)}{2(Q+2)}}\int_{\O_T}\frac{dx}{|x|^Q}\right\}^{\frac{Q+2}{2Q}}\lesssim$$
 $$\lesssim C(\log T)^{\frac{Q-2}{2Q}}\N{u}_{X_T}\left\{\left(\frac{1}{T}\right)^{\frac{Q(Q-2)}{2(Q+2)}}\log{T}\right\}^{\frac{Q+2}{2Q}} \longrightarrow0$$
 uniformly in $\l$.
 \end{proof}
%

 \begin{proof}[Proof of Proposition \ref{StimaGradiente}]
  It follows from the above lemmas.
 \end{proof}

 \section{Non degeneracy of the second differential}\label{DifferenzialeSecondo}
 
 In order to verify the non degeneracy of the second differential, we restrict ourselves to the space $\widetilde{X}_T$ defined
 in Section \ref{notazioni} (which contains $\ci{Z}_T$).
 We recall the following result \cite{MU}.
 
 \begin{prop}\label{teorMU}
  A function $u\in S^1(\H^n)$ is a solution of the following equation:
  \begin{equation}\label{NucleoDiffSecondo}
    −\D_{\H^n} u = (Q^* - 1) \o^{Q^∗-2}u
  \end{equation}
  if and only if there exist coefficients $\mi,\ni_1,\ecc,\ni_{2n}\in\R$ such that
  $$u = \mi \left. \frac{\de\o_{\l}}{\de\l} \right|_{\l=1} + \sum_{i=0}^{2n}\ni_i T_i(\o_{\l}).$$
 \end{prop}
 
  For $u$ to solve \eqref{NucleoDiffSecondo} is equivalent to being in the kernel of $\ci{J}''$.
 Since the operator $\ci{J}''$ is the sum of an isomorphism and a compact operator on $S^1(\H^n)$ (see \cite{MU})
 and that it only a negative eigenvalue whose one-dimensional eigenspace is spanned by $\o_{\l}$ (see \cite{BCD}
 there exists a constant $C$ such that if $u\in S^1(\H^n)$ and
 \begin{equation}\label{Condizioni}
  \int_{\H^n}\grad_{\H^n} u\per\grad_{\H^n}\frac{\de\o_{\l}}{\de\l} =0, \;\; \int_{\H^n}\grad_{\H^n} u\per\grad_{\H^n} T_i(\o_{\l})=0,
 \int_{\H^n}\grad_{\H^n} u\per\grad_{\H^n}\o_{\l}=0
 \end{equation}
 then
 \begin{equation}\label{NonDegGlob}
  d^2\ci{J}(\omega_{\l})[u,u] \ge C \int_{\H^n}|\grad_{\H^n} u|^2.
 \end{equation}
 
 Furthermore, since $\ci{J}''$ is selfadjoint and $\o_{\l}$ is an eigenfunction,
 \begin{equation}\label{Ortogonalita}
  d^2\ci{J}(\omega_{\l})[\o_{\l},u] =0.
 \end{equation}

 We want to use this to prove a similar non degeneracy result for $\Psi_{\l}$ on $\O_T$ for $T$ large enough.
 
 In order to do this, we introduce on $X_T$ the norm
 $$\N{u}_{T,\g{H}}^2 = \int_{\O_T}\left(|\grad_{\H^n} u|^2 + \left| \frac{u}{|x|}\right|^2\right).$$
 
 Thanks to Hardy's inequality in $\H^n$ (see Lemma 2.1 in \cite{BCX}, or otherwise apply Hölder inequality for Lorentz spaces),
 if $u\in S^1(\H^n)$, then, under the aforementioned hypotheses \eqref{Condizioni},
 $$|d^2\ci{J}(\omega_{\l})[u,u]| \ge C \int_{\H^n}|\grad_{\H^n} u|^2+ \left| \frac{u}{|x|}\right|^2.$$
 
 Using this we will prove that, if $u\in\widetilde{X}_T$ satisfies
 \begin{equation}\label{Ortogonale}
  \int_{\O_T}\grad_{\H^n} u\per\grad_{\H^n}\frac{\de\Psi_{\l}}{\de\l} =0
 \end{equation}
 and
 \begin{equation}\label{OrtogonaleAutovNeg}
  \int_{\O_T}\grad_{\H^n} u\per\grad_{\H^n}\Psi_{\l}=0,
 \end{equation}
 then, given $\e>0$, for $T$ large
 $$d^2\ci{J}_T(\Psi_{\l})[u,u] \ge C \int_{\O_T}|\grad_{\H^n} u|^2 + \left| \frac{u}{|x|}\right|^2,$$
 $$|d^2\ci{J}_T(\Psi_{\l})[\Psi_{\l},\Psi_{\l}]| \ge C \int_{\O_T}|\grad_{\H^n} \Psi_{\l}|^2 + \left| \frac{\Psi_{\l}}{|x|}\right|^2$$
 and
 $$|d^2\ci{J}_T(\Psi_{\l})[\Psi_{\l},u]| < \e \N{\Psi_{\l}}_{T,\g{H}} \N{u}_{T,\g{H}}.$$
 This implies that $\ci{J}''_T(\Psi_{\l})$ is invertible orthogonally to $\frac{\de\Psi_{\l}}{\de\l}$, and that the
 norm of the inverse is bounded uniformly in $T$.
 
 Let us take a radial function $\rho=\rho(|x|)$ such that $\rho=1$ on $\O_T$, $\rho = 0$ on $B_{1/2}\cup (\H^n\setminus B_{2T}$, $0\le\rho\le 1$, 
 $|\grad_{\H^n}\rho|\le C$ on $B_1\setminus B_{1/2}$, $|\grad_{\H^n}\rho|\le C/T$ on $B_{2T}\setminus B_{T}$.
 
 By the computations in formula \eqref{DerivataLambda} follows that
 $$\frac{\de \o_{\l}}{\de\l} = -\frac{Q-2}{2}\frac{1}{\l}\o_{\l} - \frac{1}{\l}Z(\o_{\l}).$$

 Thanks to formula \eqref{DerivataLambda}, it can easily be proved that
 \begin{equation}\label{StimaFuori}
  \left|\grad_{\H^n}\frac{\de\Psi_{\l,T}}{\de\l}\right| \le \frac{C}{\l}\frac{1}{|x|^{\frac{Q}{2}}}.
 \end{equation}
 
 By periodicity with respect to dilations we can suppose the quantity
 $$r^Q\int_{B_2\setminus B_{1/2}} \left(|\grad_{\H^n} (u\circ\d_r)|^2 + \left|\frac{u}{|x|}\circ\d_r\right|^2\right)$$
 to be minimal for $r=1$.
 Since there are $\sim\log T$ mutually disjoint annuli in $\O_T$ of the form $\d_r\{1/2 \le |x| \le 2\}$, by easy computations one gets that
 $$\int_{B_2\setminus B_{1/2}}|\grad_{\H^n} u|^2 + \left|\frac{u}{|x|}\right|^2 \le \frac{C}{\log T}\N{u}_{T,\g{H}}^2,$$
 and so, calling $W=(B_{2T}\setminus B_{T}) \cup (B_1\setminus B_{1/2})$,
 \begin{equation}\label{MinimoCilindro}
  \int_W |\grad_{\H^n} u|^2 + \left|\frac{u}{|x|}\right|^2 \le \frac{C}{\log T}\N{u}_{T,\g{H}}^2.
 \end{equation}
 
 \begin{lemma}\label{Lemma4.2}
  If $\rho$ is a cut-off function as above, for every $\e$ there exists $T_0$ such that for $T\ge T_0$ if \eqref{Ortogonale} holds then
  $$\left|\int_{\H^n}\grad_{\H^n}(\rho u)\grad_{\H^n}\frac{\de \Psi_{T,\l}}{\de\l} \right| \le \e \frac{1}{\l}\N{u}_{T,\g{H}}$$
  and
  $$\left|\int_{\H^n}\grad_{\H^n}(\rho u)\grad_{\H^n}\Psi_{T,\l}\right| \le \e \N{u}_{T,\g{H}}.$$
 \end{lemma}

 \begin{proof}
  $$\int_{\H^n}\grad_{\H^n}(\rho u)\grad_{\H^n}\frac{\de \Psi_{\l}}{\de\l} =
  \int_{\H^n}\grad_{\H^n}(\rho u)\grad_{\H^n}\frac{\de \Psi_{\l}}{\de\l}-\int_{\O_T}\grad_{\H^n} u\per\grad_{\H^n}\frac{\de\Psi_{\l}}{\de\l} =$$
  $$ =  \int_W\left[(\rho\grad_{\H^n}u + u\grad_{\H^n}\rho) \grad_{\H^n}\frac{\de \Psi_{\l}}{\de\l}\right].$$
  Thanks to formulas \eqref{StimaFuori} and \eqref{MinimoCilindro} the first estimate follows by easy computations.
  The proof of the second one is identical.
 \end{proof}
 
 \begin{lemma}\label{Lemma4.3}
  For every $\e$ there exists $T_0$ such that for $T\ge T_0$ if \eqref{Ortogonale} holds then
  $$\int_{\H^n}\grad_{\H^n}(\rho u)\grad_{\H^n}\frac{\de \o_{\l}}{\de\l} \le \e \frac{1}{\l} \N{u}_{T,\g{H}}$$
  and
  $$\int_{\H^n}\grad_{\H^n}(\rho u)\grad_{\H^n}\o_{\l} \le \e \N{u}_{T,\g{H}}.$$
 \end{lemma}

 \begin{proof}
  Thanks to Lemma \ref{Lemma4.2}, we can estimate
  $$\int_{\H^n}\grad_{\H^n}(\rho u)\l\grad_{\H^n}\frac{\de \Psi_{\l}}{\de\l} - \int_{\H^n}\grad_{\H^n}(\rho u)\l\grad_{\H^n}\frac{\de \o_{\l}}{\de\l}\le$$
  $$\le C \N{u}_{T,\g{H}} \left( \int_{\O_T\cup W}\left|\l\grad_{\H^n}\frac{\de \Psi_{\l}}{\de\l}-\l\grad_{\H^n}\frac{\de \o_{\l}}{\de\l}\right|^2\right)^{1/2}.$$
  This quantity can be estimated almost identically as in the proof of Lemma \ref{TermineB}.
  The proof of the second inequality estimate is identical.
 \end{proof}
 
 \begin{lemma}
  For every $\e>0$ there exist constants $T_0$ and $C$ such that for $T\ge T_0$ if \eqref{Ortogonale} and \eqref{OrtogonaleAutovNeg} hold then
  $$|d^2\ci{J}(\omega_{\l})[\rho u,\rho u]| \ge C \int_{\H^n}|\grad_{\H^n} (\rho u)|^2 + \left|\frac{\rho u}{|x|}\right|^2,$$
  $$|d^2\ci{J}(\omega_{\l})[\rho \Psi_{\l},\rho \Psi_{\l}]| \ge
  C \int_{\H^n}|\grad_{\H^n} (\rho \Psi_{\l})|^2 + \left|\frac{\rho \Psi_{\l}}{|x|}\right|^2$$
  and
  $$|d^2\ci{J}(\omega_{\l})[\rho \Psi_{\l},\rho u]| \le \e \N{\Psi_{\l}}_{T,\g{H}} \N{u}_{T,\g{H}}.$$
 \end{lemma}

 \begin{proof}
  Since $u\in\widetilde{X}_T$, $u\rho$ is invariant with respect to the symmetry $(x,t)\mapsto (-x,t)$, one has
  $$\int_{\H^n}\grad_{\H^n} (\rho u)\per\grad_{\H^n} T_i(\o_{\l})=0.$$
  The claim follows by Lemma \ref{Lemma4.3}, by equations \eqref{NonDegGlob} and \eqref{Ortogonalita}, and elementary linear algebra.
 \end{proof}
 
 \begin{lemma}
  For every $\e>0$ there exist constants $T_0$ and $C$ such that for $T\ge T_0$ if conditions \eqref{Ortogonale} and \eqref{OrtogonaleAutovNeg} hold,
  then
  $$|d^2\ci{J}_T(\Psi_{\l})[u,u]| \ge C \int_{\O_T}|\grad_{\H^n} u|^2,$$
  $$|d^2\ci{J}_T(\Psi_{\l})[\Psi_{\l},\Psi_{\l}]| \ge C \int_{\O_T}|\grad_{\H^n} \Psi_{\l}|^2$$
  and
  $$|d^2\ci{J}_T(\Psi_{\l})[\Psi_{\l},u]| < \e \N{\Psi_{\l}}_{X_T} \N{u}_{X_T}.$$
 \end{lemma}

 \begin{proof}
  By direct computation we find
  $$\left| d^2\ci{J}(\omega_{\l})[\rho u,\rho u] - d^2\ci{J}_T(\Psi_{\l})[u,u]\right|  =$$
  $$=\left|\int_{\H^n}|\grad_{\H^n}(\rho u)|^2 -(2^*-1)|\omega_{\l}|^{2^*-2}\rho^2u^2 +\right.$$
  $$ \left.- \int_{\Omega_T}|\grad_{\H^n} u|^2  -(2^*-1)|\Psi_{\l}|^{2^*-2}u^2\right|\le$$
  $$\le (2^*-1)\left|\int_{\Omega_T}\left(|\Psi_{\l}|^{2^*-2}-|\omega_{\l}|^{2^*-2}\right)u^2\right|+$$
  $$+ (2^*-1) \left|\left( \int_{B_{2T}\setminus B_{T}} + \int_{B_1\setminus B_{1/2}}\right)|\omega_{\l}|^{2^*-2}\rho^2u^2\right| +$$
  $$+ 2\left|\left( \int_{B_{2T}\setminus B_{T}} + \int_{B_1\setminus B_{1/2}}\right) (u^2|\grad_{\H^n}\rho|^2+\rho^2|\grad_{\H^n} u|^2)\right| .$$
  The first term can be estimated as in Lemma \ref{TermineB}, the second in a trivial way, and the third has been essentially already
  estimated, to prove that for every $\e$
  there exists $T$ big enough to ensure that the whole sum is bounded by $\e \N{u}_{X_T}^2$.
 
  Analogously
  $$ \left| \int_{\H^n}|\grad_{\H^n} (\rho u)|^2 - \int_{\O_T}|\grad_{\H^n} u|^2\right| \le \e \N{u}^2_{X_T}.$$
  This implies the first part of the thesis. The other statements are deduced in an analogous manner.
 \end{proof}
 
 \begin{prop}\label{NonDegDiffSec}
  There exist constants $T_0$ and $C$ such that for $T\ge T_0$ the operator $\ci{J}''_T(\Psi_{\l})$
  is invertible on the orthogonal space of $\frac{\de\Psi_{\l}}{\de\l}$ in $X_T$, and
  $\N{\ci{J}''_T(\Psi_{\l})^{-1}}_{\ci{L}(X_T)}\le C$.
 \end{prop}

 \begin{proof}
  It follows from the preceding lemmas and elementary Hilbert space theory.
 \end{proof}
 
 \section{Proof of the main Theorem}
 We have proved that, for $T$ big enough, on the orthogonal in $\widetilde{X}_T$ of the tangent of the curve $\ci{Z}_T$ the second differential of $\ci{J}_T$ is non degenerate,
 with norm bounded independently by $\l$ and $T$.
 Let us call $W$ this orthogonal in the point $\Psi_{\l}\in\ci{Z}$ and $\pi$ the orthogonal projection on $W$. We remember that our aim is to solve $\grad_{\H^n}\ci{J}_T(u)=0$.
 Following the standard reasoning in \cite{AM} we note that this is equivalent to solve
 $$\pi\grad_{\H^n}\ci{J}_T(\Psi_{\l}+w)=0$$
 (auxiliary equation) and
 $$(I-\pi)\grad_{\H^n}\ci{J}_T(\Psi_{\l}+w)=0$$
 (bifurcation equation) with $w\in W$.
 
 \begin{lemma}
  There exists $T_0$ such that the auxiliary equation has a unique solution $w_T(\l)$; furthermore $\sup_{\l}\N{w_T(\l)}\to 0$ for $T\to\infty$.
 \end{lemma}

 \begin{proof}
  Write
  $$\grad_{\H^n}\ci{J}_T(\Psi_{\l}+w) = \grad_{\H^n}\ci{J}_T(\Psi_{\l}) + \ci{J}''_T[w] + R(\Psi_{\l},w)$$
  with $R(\Psi_{\l},w)=o(\N{w})$ and $R(\Psi_{\l},w)-R(\Psi_{\l},v)=o(\N{w-v})$, so that the auxiliary equation becomes
  $$\pi\grad_{\H^n}\ci{J}_T(\Psi_{\l}) + \pi\ci{J}''_T(\Psi_{\l})[w] + \pi R(\Psi_{\l},w)= 0,$$
  namely
  $$w=-(\pi\ci{J}''_T(\Psi_{\l}))^{-1}\left[\pi\grad_{\H^n}\ci{J}_T(\Psi_{\l}) + \pi R(\Psi_{\l},w)\right]:= N_{\l}(w).$$
  By Propositions \ref{StimaGradiente} and \ref{NonDegDiffSec}, $N$ is a contraction if $T$ is big enough,
  and so the auxiliary equation has an unique solution $w=w_T(\l)$.
  Furthermore for every $r>0$ there exists $T$ big enough such that $B_r(\Psi_{\l})\cap W$ is mapped into itself by $N$.
  So $\sup_{\l}\N{w_T(\l)}$ tends to zero for $T\to\infty$.
 \end{proof}
 
 \begin{proof}[Proof of Theorem \ref{Teorema}]
  Let us consider the function
  $$\Phi(\l)= \ci{J}_T(\Psi_{\l}+w(\l)).$$
  It is continuous and periodic, so it a stationary point $\l_0$.
  Following the standard argument of Theorem 2.12 and Remark 2.14 in \cite{AM}, with the need for only formal modifications,
  the fact that $\Phi'(\l_0)= \ci{J}_T'(\Psi_{\l_0}+w(\l_0))\per(\frac{\de\Psi_{\l_0}}{\de\l}+w'(\l_0))$ implies $u=\Psi_{\l_0}+w(\l_0)$
  to solve the bifurcation equation, and so to be a stationary point of $\ci{J}_T$.
  
  The smoothness of the solution can be proved with the same method of Appendix B in \cite{Str}.

 Also $\l^{(2-Q)/2}u\circ\d_{\l^{-1}}$ is a critical point of $\ci{J}_T$, and by the unicity in the fixed point theorem it must be equal to
 $\Psi_{\l_0\l}+w(\l_0\l)$, and so the whole curve $\widetilde{\ci{Z}}_T=\{ \Psi_{\l}+w(\l)\}$ consists of critical points of $\ci{J}$.
 

 To prove the positivity, let us notice that from the proof of Proposition \ref{NonDegDiffSec} follows that
 $\ci{J}(\o_{\l})$ has Morse index one on $\left\{ \l\frac{\de\o_{\l}}{\de \l}\right\}^{\perp}$.
 By continuity, the same holds for the orthogonal to the tangent space to $\widetilde{\ci{Z}}_T$. Since $d\ci{J}_T$ is zero on $\widetilde{\ci{Z}}_T$, the tangent of $\widetilde{\ci{Z}}_T$ is in the kernel of $\ci{J}''_T$.
 So the Morse index of $\ci{J}_T$ on $\widetilde{X}_T$ is one.

 By a slight adaptation of the proof of Proposition 3.2 in \cite{BCD} the set $\{u\ne 0\}$ has at most one connected component modulo $\d_T$, and so
 $u$ does not change sign. By construction it is evident that it must be weakly positive (and even if it was not, it would be enough to change sign).
 The strict positivity follows from Bony's maximum principle (see \cite{Bon}).
 
 The last assertion follows by construction.
 \end{proof}

\end{document}